\documentclass{article}
\usepackage{amsmath}
\usepackage{amsthm}
\usepackage{graphics}
\usepackage{epstopdf}
\usepackage{amssymb}
\usepackage{txfonts}
\usepackage[latin5]{inputenc}
\usepackage[T1]{fontenc}
\usepackage{setspace}
\usepackage{url}
\usepackage{verbatim}
\numberwithin{equation}{section}

\theoremstyle{plain}
\newtheorem{thm}{Theorem}[section]
\newtheorem{lem}{Lemma}[section]
\newtheorem{prop}{Proposition}[section]

\newtheorem{rem}[thm]{Remark}
\newtheorem{defn}{Definition}[section]
\newtheorem{eg}{Example}[section]

\title{Using Tropical Degenerations For Proving The Nonexistence Of Certain Nets}
\author{Mustafa Hakan G\"UNT\"URK\"UN \\ hakan.gunturkun@gediz.edu.tr  \\ \\ Ali Ula\c{s} \"Ozg\"ur K\.I\c{S}\.ISEL \\ akisisel@metu.edu.tr}
\date{}

\begin{document}
\maketitle
	\begin{abstract} A net is a special configuration of lines and points in the projective plane. There are certain restrictions on the number of its lines and points. We proved that there cannot be any (4,4) nets in $\mathbb{C}P^2$. In order to show this, we use tropical algebraic geometry. We tropicalize the hypothetical net and show that there cannot be such a configuration in  $\mathbb{C}P^2$.

	\end{abstract}

\section{Introduction}
	A finite hyperplane arrangement is a finite set of hyperplanes in a projective space over a field. If the space is the projective plane, then the arrangement is called a line arrangement. A net is a special line configuration in the projective plane. There are some restrictions on the structure of nets discovered by S. Yuzvinsky and some open problems remain.

One way to understand tropical algebraic geometry is by looking at certain limits of complex algebraic varieties under the logarithm map. It may be easier to deal with the tropical counterparts of the classical problems because we can use combinatorics extensively on these simpler objects. 

In this note we show that there cannot be any (4,4)-nets in $\mathbb{C}P^2$. To show this we take a hypothetical net, we tropicalize its lines and points. We draw some of them on the tropical plane. Then by using the tropical picture and intersection relations, we find the possible locations of some of the other lines and points. This leads to a contradiction which shows the nonexistence of (4,4)-nets in $\mathbb{C}P^2$.

\section{Preliminaries}
\subsection{Nets}

 \begin{defn}
 
 	Let $k>1$ be a positive integer and $\mathbb{P}^2$ the complex projective plane. Let $\mathcal{A}_i$ be a finite set of lines for each $i \in\{1\ldots k\}$, and  $\mathcal{X}$ a finite set of points.

        The collection $(\mathcal{A}_1, \ldots , \mathcal{A}_k, \mathcal{X})$ is called a $\textbf{k-net}$ if the following are satisfied:
     
        \begin{enumerate}
          \item When $i\neq j$, $\mathcal{A}_i$ and $\mathcal{A}_j$ are disjoint.
          \item If $i\neq j$,  $\ell \in \mathcal{A}_i$ and $m\in \mathcal{A}_j$ then $\ell \cap m \in \mathcal{X}$.
          \item For every $p\in \mathcal{X}$ and $i\in \{1, \ldots, k\}$ there exists a unique $\ell \in \mathcal{A}_i$ such that $p\in \ell$.
        \end{enumerate} 
        
        \end{defn}
        
\begin{eg} The following is an example of a 3-net. 
	$$\mathcal{A}_1=\{\ell_{11}, \ell_{12}\}$$
	$$\mathcal{A}_2=\{\ell_{21}, \ell_{22}\}$$
	$$\mathcal{A}_3=\{\ell_{31}, \ell_{32}\}$$
	$$\mathcal{X}=\{p_{11}, p_{12}, p_{21}, p_{22}\}$$
	The points are indexed according to the rule $\ell_{1i}\cap \ell_{2j} = p_{ij}$

	\begin{figure}[h] \begin{center} \resizebox{8cm}{6cm} {\includegraphics{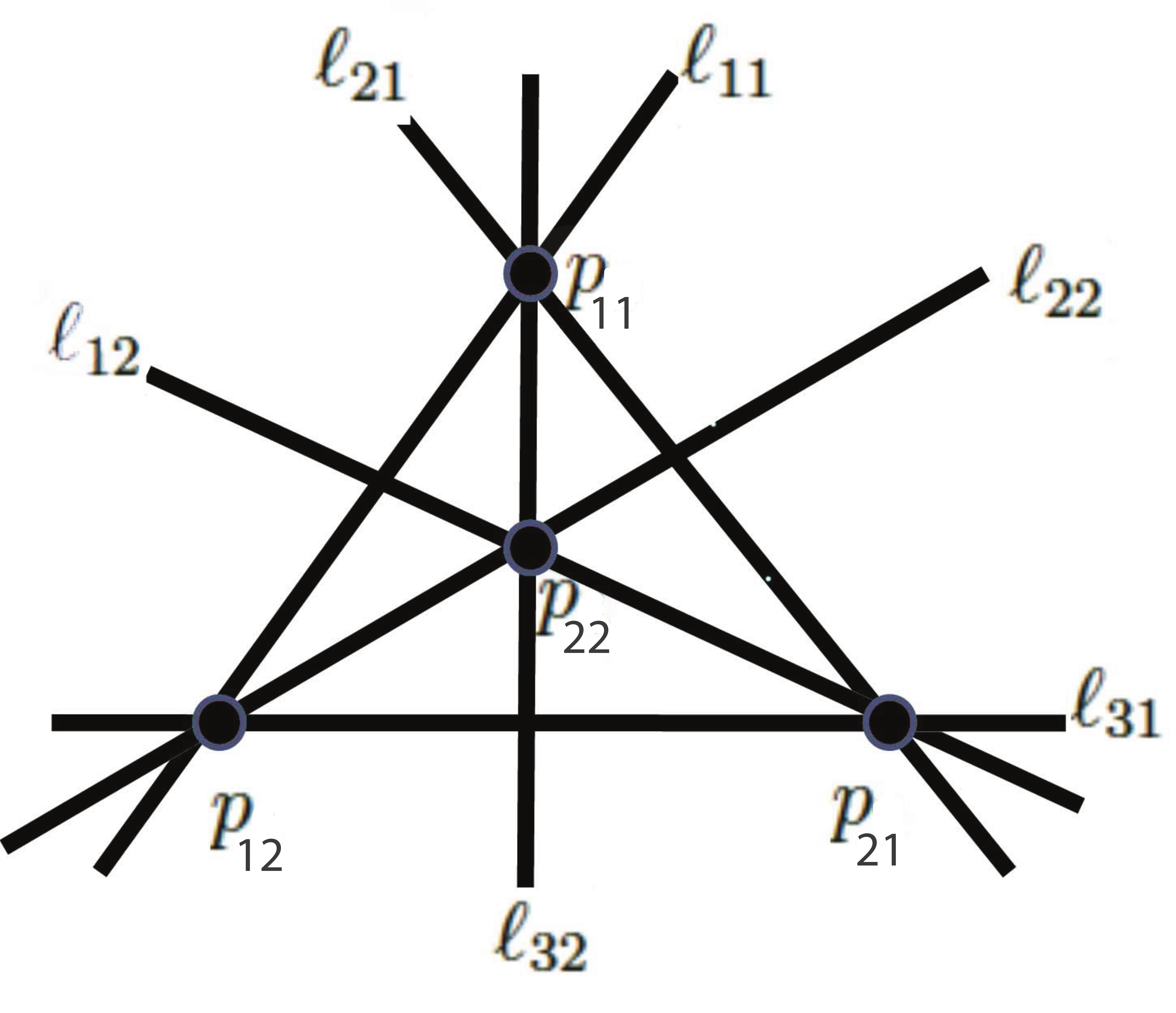}} \caption{A 3-Net} \label{fig:anet} \end{center}\end{figure}
\end{eg}


\subsubsection{Some Properties}

\begin{prop}Let $k\geq 3$. $(\mathcal{A}_1,\ldots,\mathcal{A}_k, \mathcal{X})$ be a $k$-net. Then the following properties hold:
		\begin{enumerate}
			\item $|\mathcal{A}_i|=|\mathcal{A}_j| \ \forall i,j$
          		\item $| \mathcal{X} | = |\mathcal{A}_1|^2$
		\end{enumerate}
        \end{prop}
        
        \begin{proof} Let $\mathcal{A}_i$ and $\mathcal{A}_j$ be two of the line sets, $i\neq j$.
        		\begin{enumerate}
			\item By the second statement of \textit{Definition 2.1} all the intersections of the lines in $\mathcal{A}_i$ and $\mathcal{A}_j$ are in $\mathcal{X}$, hence $\mathcal{A}_i \cap \mathcal{A}_j \subseteq \mathcal{X} $. On the other hand each element of $\mathcal{X}$ is an element of $\mathcal{A}_i \cap \mathcal{A}_j$, hence $\mathcal{X} \subseteq \mathcal{A}_i \cap \mathcal{A}_j$. So $\mathcal{X} = \mathcal{A}_i \cap \mathcal{A}_j$. Furthermore note that $|\mathcal{A}_i \cap \mathcal{A}_j| = |\mathcal{A}_i||\mathcal{A}_j|$ by the third statement of \textit{Definition 2.1}. Therefore $ |\mathcal{X}|=|\mathcal{A}_i||\mathcal{A}_j|$. Since $k\geq 3$, choose $\ell \neq i,j$. But then $|\mathcal{A}_\ell||\mathcal{A}_i|=|\mathcal{A}_j||\mathcal{A}_\ell|$ which means that $|\mathcal{A}_i|=|\mathcal{A}_j| \ \forall i,j$.
          		\item Using $|\mathcal{X}|=|\mathcal{A}_i||\mathcal{A}_j|$ and $|\mathcal{A}_i|=|\mathcal{A}_j|=|\mathcal{A}_1|$, we get $| \mathcal{X} | = |\mathcal{A}_1|^2$.
		\end{enumerate}
        \end{proof}
        
        If $|\mathcal{A}_i|=d$ then $|\mathcal{X}| = d^2$. From here on we shall use the phrase ``$(k,d)$-net'' instead of ``$k$-net''. For instance the 3-net above will be called a (3,2)-net.

\subsubsection{Pencils of curves and $(k,d)$-nets}

        \begin{prop} \cite{yuz}
        Let $\{\mathcal{A}_i\}^{k}_{i=1}$ be disjoint sets of lines each of which includes d different lines, $|\mathcal{A}_i \cap \mathcal{A}_j| = d^2$ $for \ all \ i\neq j$, and $\{C_i\}^{k}_{i=1}$ be the curves of degree d formed by the union of lines in each set. Then the $\mathcal{A}_i$ are the sets of lines of a $(k,d)$-net if and only if  the $C_i$ are the fibers of a pencil formed by any two of them.
        \end{prop}
        
	\begin{proof}
	The proof follows from Noether's AF+BG theorem. See \textit{Lemma 3.1} of \cite{yuz}.
	\end{proof}
	
	\begin{eg}
	The \textit{Figure 2} is a $(3,2)$-net. Suppose the equations of the lines involved are: 
	
	\begin{eqnarray*}
	\ell_{11} : y-1=0 \\
	\ell_{12} : y+1=0 \\
	\ell_{21}: x+1=0 \\
	\ell_{22}: x-1=0 \\
	\ell_{31}:y-x=0 \\
	\ell_{32}:y+x=0 \\
	\end{eqnarray*}
	
	\begin{figure} \begin{center} \resizebox{8cm}{6cm}{\includegraphics{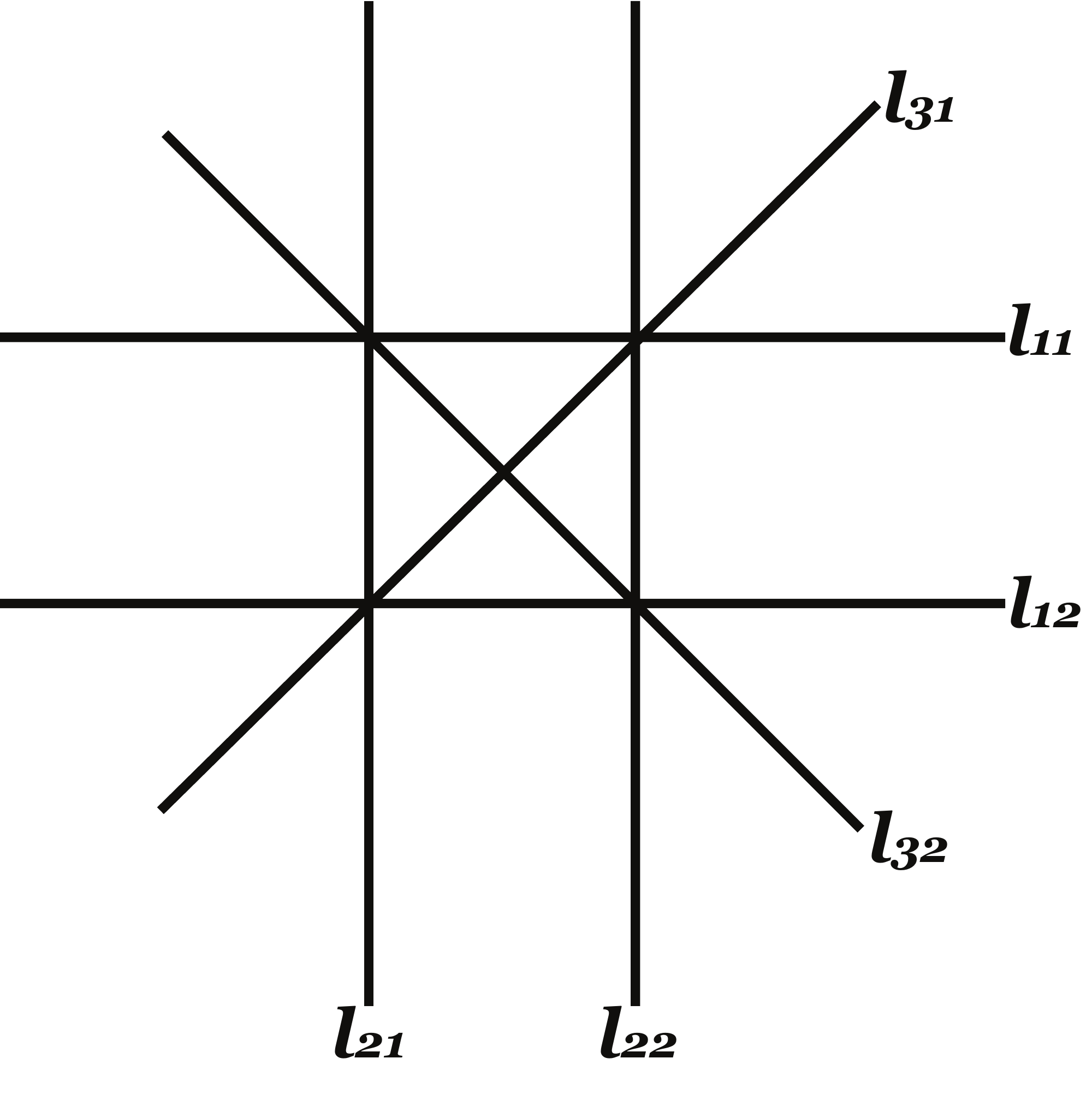}} \caption{A (3,2)-Net} \label{fig:pencil} \end{center}\end{figure}
	
	If we take any two of the classes and form a pencil the other will be an element of that pencil. For instance if we take $\lambda (x^2-1)+\mu (y^2-1)=0$ as the pencil, then  \\ $[\lambda: \mu ] \in \{[1:0], [0:1],[1:-1]\}$ give the fibers. We show this in \textit{Figure 3}.
	
	\begin{figure} \begin{center} \resizebox{8cm}{6cm}{\includegraphics{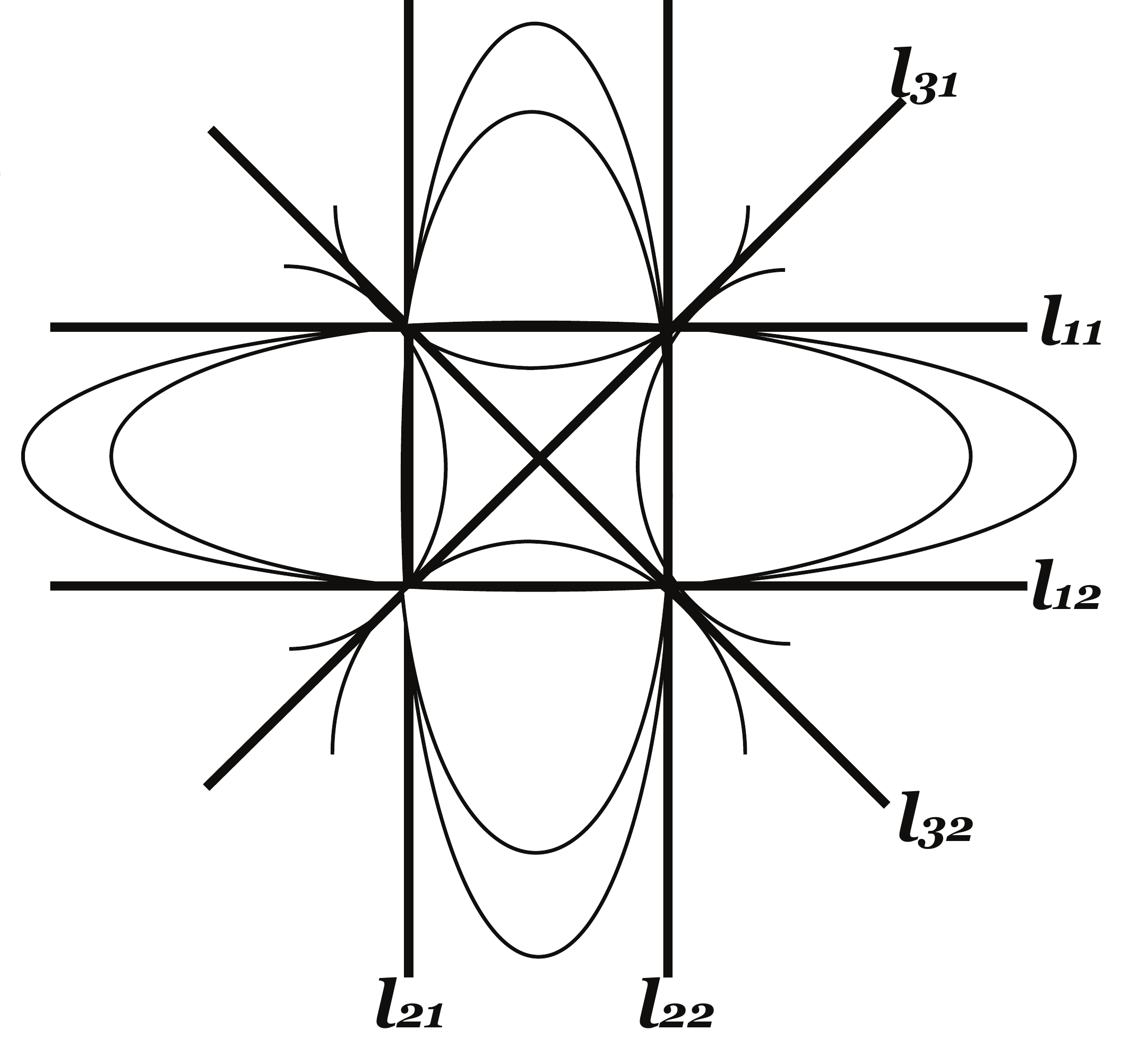}} \caption{A Pencil of Quadrics} \label{fig:pencil2} \end{center}\end{figure}
	\end{eg}
	
	\newpage
	
\subsubsection{The restrictions on $(k,d)$-nets}

	There does not exist any 1-net because of the definition of the net. If we take two sets including arbitrary numbers of lines and if we take all the intersections as the point set, it forms a $(2,d)$-net provided that no 3 of these lines are concurrent. Therefore the $k=2$ case is trivial. For $k=3$, $d=1$ we would have $\mathcal{X}=\{p\}$, $\mathcal{A}_i = \{\ell_i\}$ and $\ell_1, \ldots, \ell_k$ are concurrent. 
	 
        \begin{thm}{(S.Yuzvinsky)\cite{yuz}}: If a $(k,d)$-net where $d>1$ exists in $\mathbb{P}^2$ then $(k,d)$ must be one of the following:
        \begin{itemize}
          \item $k=3, \ d\geq 2$
          \item $k=4,\ d\geq 3$
          \item $k=5,\ d\geq 6$
        \end{itemize} \end{thm}
        
       \begin{proof} See $Theorem\ 3.2$ at \cite{yuz}. \end{proof}      
       
       \subsubsection{The Main Problem}
       
	We can find $(3,d)$-nets for every $d$. For the construction of $(3,d)$-nets see \textit{Proposition 3.3} of \cite{yuz}.  J. Stipins proved in his dissertation that there cannot be any 5-nets \cite{stipins}. Other than these, there is only one 4-net known which is the $(4,3)$-net. In this note we showed that there cannot be any $(4,4)$-nets in $\mathbb{C}P^2$. We used tropical geometry to solve this problem.

\subsubsection{Latin Squares and $(k,d)$-nets}
	
		\begin{defn}[Latin Square] A $d\times d$ matrix such that there exists a bijection between each row and each column and the set $\{ 1,\ldots, d\}$ is called a $d\times d$ \textbf{Latin square}.\end{defn}
		
		\begin{defn}[Orthogonal Pairs] Let $\mathcal{L},\mathcal{L}'$ be two Latin squares. If one can find a bijection between the sets $\{1\ldots d \}\times \{1\ldots d \}$ and the set of pairs $\{(\mathcal{L}_{ij},\mathcal{L}'_{ij})\}$ then the pair of Latin squares ($\mathcal{L}$, $\mathcal{L}'$) is called an \textbf{orthogonal pair}.\end{defn}
		
		\begin{defn}[Orthogonal Set] A set of Latin squares  $\{\mathcal{L}_1,\ldots , \mathcal{L}_n\}$ such that ($\mathcal{L}_i$ , $\mathcal{L}_j$) is an orthogonal pair for all $1\le i<j \le n$ is called an \textbf{orthogonal set}. \end{defn}
	
	The following two propositions below are taken from \textit{Chapter 2} of \cite{stipins}.
	
	\begin{prop} Let $(\mathcal{A}_1,\ldots,\mathcal{A}_k, \mathcal{X})$ be a $(k,d)$-net. Then the set of Latin squares $\{\mathcal{M}_3,\ldots , \mathcal{M}_k\}$ below is an orthogonal set:
           $$l_{1i},l_{2j} \text{ and }\ l_{t(\mathcal{M}_t)_{ij}}\ pass\ through\ the\ same\ point. \ (*) \ where\ 3\leq t \leq k $$ 
	\end{prop}
        
    The converse of the above is also true:

        \begin{prop} Let $\mathcal{A}_1=\{l_{11}, \ldots , l_{1d}\}, \mathcal{A}_2=\{l_{21}, \ldots , l_{2d}\}$ be two sets containing $d$ lines intersecting at $d^2$ points. Let $\mathcal{X} = \mathcal{A}_1 \cap \mathcal{A}_2$,\ $\{\mathcal{M}_3,\ldots , \mathcal{M}_k\}$ be an orthogonal Latin square set. Suppose that $(\mathcal{A}_1,\ldots,\mathcal{A}_k)$ be the sets of lines satisfying the incidence relations (*). Then $(\mathcal{A}_1,\ldots,\mathcal{A}_k, \mathcal{X})$ forms a $(k,d)$-net.\end{prop}
        
    Therefore an orthogonal set of Latin squares can be thought of as defining an abstract $(k,d)$-net. Our problem is to determine whether an abstract $(k,d)$-net can be embedded in $\mathbb{C}P^2$. For some values of \textit{k} and \textit{d}, even an abstract $(k,d)$-net is impossible to find. For example we know by the famous Euler's conjecture on Latin squares (stated by Euler in 1779, proved by Gaston Tarry in 1900 \cite{tarry}) that there are no $6\times 6$ orthogonal pairs of Latin squares. Because of this there are no $(4,6)$-nets. However the $(k,d)$-nets we interested in (and the nets which have restrictions above) are not abstract nets, they are nets which can be seen in $\mathbb{C}P^2$. So that, although the Latin squares determine the existence of the abstract nets, we should use other methods to find whether or not a net can be realized in $\mathbb{C}P^2$. Let us finish this section by the next proposition which is straightforward to prove.

        \begin{prop} $\mathcal{X}=\{p_1,\ldots,p_{d^2}\}$. The elements of $\mathcal{X}$ are the points of a $(k,d)$-net if and only if $\mathcal{X}$ can be partitioned in $k$ different ways into $d$ sets each of which includes $d$ collinear points. \end{prop}     

\subsection{Tropical Lines}

	There are many ways to describe tropical curves, and in particular tropical lines \cite{gath} \cite{spe}. We choose the pathway using amoebas of curves.
       	
	\begin{defn}\cite{mik} Let $V \subset  (\mathbb{C}^*)^n$ be an algebraic variety where $\mathbb{C}^* = \mathbb{C} - \{0\}$
                        \begin{eqnarray}
                            Log:(\mathbb{C}^*)^n &\rightarrow& \mathbb{R}^n \nonumber \\
                            (z_1,\ldots,z_n)&\mapsto& (log|z_1|,\ldots,log|z_n|)\nonumber
                        \end{eqnarray}
                       Then the set $Log(V)$ is called the \textbf{amoeba} of $V$.
	\end{defn}                      
				
	\begin{prop}
		If $V=\{z_1+z_2=-1\}$ then the graph of $Log(V)$ is as in \textit{Figure 4} \cite{iten}.
		
			\begin{figure}[h]\begin{center} \resizebox{8cm}{5.7cm}{\includegraphics{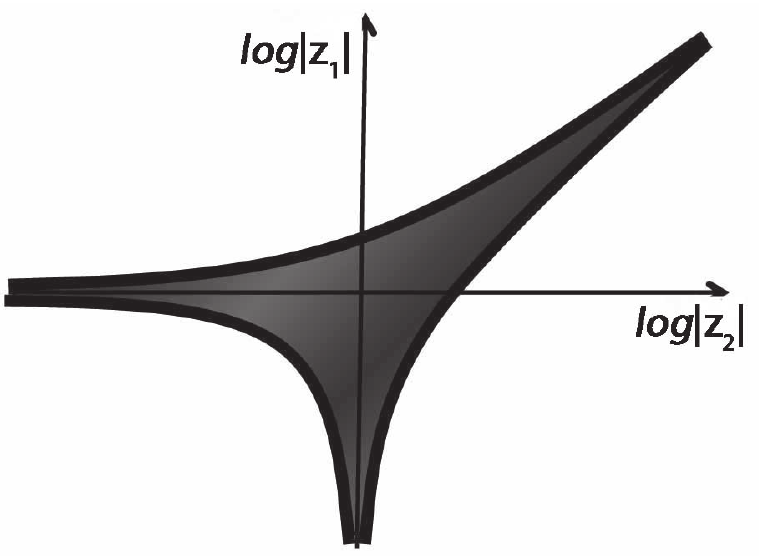}} \caption{The Amoeba of $V$} \label{fig:amip} \end{center}\end{figure}

	\end{prop} 
	
	\begin{proof}
	$z_1=r_1e^{i\theta_1}$ and $z_2=r_2e^{i\theta_2}$, $r_1=|z_1|$ and $r_2=|z_2|$.
	A point $(x,y)$ with $x=log(r_1)$ and $y=log(r_2)$ belongs to the amoeba if and only if there exist $\theta_1$, $\theta_2$ such that $r_1e^{i\theta_1}+r_2e^{i\theta_2}=-1$.
	By the triangle inequality, the boundaries of the amoeba correspond to $r_2-r_1=1$, $r_1-r_2=1$ and $r_1+r_2=1$.
	We check the boundaries one by one: \\
	$\underline{r_2-r_1=1}\Rightarrow$ $e^y-e^x=1$: We solve this equality as $y=log(1+e^x)$. We see that \textit{y} increases with \textit{x}, and $\mathop{\lim} \limits_{x \to \infty} (y) \to \infty$. The graph of $y=log(1+e^x)$ is asymptotic to $y=x$.  Also $\mathop{\lim} \limits_{x \to -\infty} (y)= 0$ which explains the boundary in the second quadrant.\\
	$\underline{r_1-r_2=1}\Rightarrow$ $e^x-e^y=1$: We may obtain this graph by changing \textit{x} and \textit{y} in the above case which means that the graph of $e^x-e^y=1$ is the symmetric to the graph of $e^y-e^x=1$ with respect to $x=y$.\\
	$\underline{r_1+r_2=1} \Rightarrow$ $e^x+e^y=1$: $y=log(1-e^x)$. \textit{y} decreases if \textit{x} increases. This function is defined when $e^x<1$, that is $x<0$. If $x<0$ then $y<0$. $\mathop{\lim} \limits_{x \to 0^{-}} (y)\to -\infty$ and $\mathop{\lim} \limits_{x \to -\infty} (y)\to 0$. This explains the lower left boundary. 
	\end{proof}
		
	\begin{prop}
	If $V=\{az_1+bz_2+c=0\}$ then the graph of the amoeba of $V$ is the translation of \textit{Figure 4} by $log{\frac{c}{a}}$ and $log{\frac{c}{b}}$ in the directions of $log|z_1|$ and $log|z_2|$ respectively.
	\end{prop}
	
	\begin{proof}
	Since $log|az_1|=log|a|+log|z_1|$, the effect of $a$ is translation of the figure in the direction of the negative $log|z_1|$ axis by $log|a|$. Similarly the effect of $b$ is translation of the figure in the direction of the negative $log|z_2|$ axis by $log|b|$. The effect of $c$ is translation of the figure in the direction of the positive $log|z_1|$ axis and the positive $log|z_2|$ axis by $log|c|$. The overall effect is translation by $log|c|-log|a|$ and  $log|c|-log|b|$ in the directions of $log|z_1|$ and $log|z_2|$ axes respectively.
	\end{proof}
				
	\begin{defn}				
		Let $V_t\subset \mathbb{(C^*)}^n$ be a one parameter family of subvarieties of $\mathbb{(C^*)}^n$. The set $\mathop{\lim} \limits_{t \to \infty} (log_t(V_t))$ is called the \textbf{tropicalization} of $V_t$. If $V_t$ is a family of lines in $(\mathbb{C^*})^2$ the graph of $\mathop{\lim} \limits_{t \to \infty} (log_t(V_t))$ is called a \textbf{tropical line}.
	\end{defn}
	
	\begin{prop}
		If $V_t=\{az_1+bz_2=c\}$ where $a,b,c\in \mathbb{C}$ then the tropical line $\mathop{\lim} \limits_{t \to \infty} (log_t(V_t))$ is in \textit{Figure 5}.
		
		\begin{figure}[h] \begin{center} \resizebox{6cm}{4.3cm}{\includegraphics{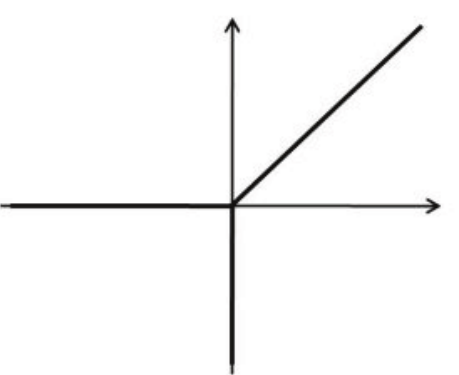}} \caption{A Tropical Line} \label{fig:tropicalline} \end{center}\end{figure}
		
	\end{prop}
	
	\newpage
	
	\begin{proof}
	$log_t|az_1|=log_t|a|+log_t|z_1|=\frac{log|a|}{logt}+\frac{log|z_1|}{logt}$.
	If ${t \to \infty}$ then $\frac{log|a|}{logt} \to 0$ so the number $a$ has no impact. Similarly the numbers $b$ and $c$ have no impact. The effect of $logt$ on $log|z_1|$ and $log|z_2|$ is shrinking the figure. Let us consider what happens to the boundary curves. \\ 
	The upper boundary curve becomes $(logt)y=log(1+e^{(logt)x})$. This function is increasing and above the proposed limit curve. We want to find the \textit{y-intercept} of the graph. If $x=0$, that is, $log{r_1}=0$ then $r_1=1$. Since $z_1+z_2=-1$, $z_2$ may be at most 2, therefore $y=log_t{2}$. Therefore the $y-intercept$ approaches 0 as $t \to \infty$ and, $y''=logt\frac{e^{(logt)x}}{((1+e^{(logt)x})^2}$ which is always greater than 0 where \textit{t} is big enough. Therefore the graph of the upper boundary curve is concave up with the $y-intercept$ shrinking to 0. Notice that $y=\frac{log(1+e^{(logt)x})}{logt}$ is asymptotic to $x$ as $x \to \infty$ and to $y=0$ as $x \to -\infty$.\\
	The boundary curve on the right is similar. The \textit{x-intercept} is $(0,log_t{2})$ and $y''<0$.\\
	The lower boundary curve would be $(logt)y=log(1-e^{(logt)x})$. This function is decreasing and below the proposed limit curve. We have $x,y<0$, therefore $r_1,r_2<1$. The equation is $z_1+z_2=-1$, the point on $x=y$ on the boundary is $(log\frac{1}{2},log\frac{1}{2})=(-log_t2,-log_t2)$. This point aproaches $(0,0)$ as $t \to \infty$. We have $y''=-logt\frac{e^{(logt)x}}{((1+e^{(logt)x})^2}$, which is smaller than 0 for big values of \textit{t}. Therefore the graph is concave down. Notice that $y=\frac{log(1-e^{(logt)x})}{logt}$ is asymptotic to $y=0$ as $x\to -\infty$, and to $-\infty$ as $x=0$.
	\end{proof}
	
	If the coefficients are just numbers then the tropical line always has a center at origin. In order to get nontrivial tropical lines, instead of looking at just one variety, we look at families of varieties. Hence we change the coefficients to polynomials in \textit{t}. In the next proposition we will see the effect of these polynomials to the tropical line.\\
	
	\begin{prop}
	Let $V_t=\{ f(t)z_1+g(t)z_2=h(t)\}$ be a family of lines in $(\mathbb{C^*})^2$ where f(t), g(t), h(t) are polynomials and $n_f,\ n_g$ and $n_h$ are the degrees of $f,\ g$ and $h$ respectively. Then the graph of the tropical line $\mathop{\lim} \limits_{t \to \infty} (log_t(V_t))$ is the translation of \textit{Figure 5} by $n_h-n_f$ and $n_h-n_g$ in the directions of $log|z_1|$ and $log|z_2|$ axes respectively.
	\end{prop}
	 
	\begin{proof}
	The degree of \textit{f(t)} is $n_f$. The other powers of \textit{f(t)} does not have an effect on $\mathop{\lim} \limits_{t \to \infty} (log_t(f(t)))$, hence the graph is only effected by $n_f$. Similarly the graph is effected by $n_g$ and $n_h$. The rest  is similar to the proof of \textit{Proposition 2.7}.
	\end{proof}
	
	\begin{prop}
	Let $V_t$ and $W_t$ be two families of lines in $(\mathbb{C^*})^2$. \\ Then $\mathop{\lim} \limits_{t \to \infty} (log_t(V_t\cup W_t))$ = $\mathop{\lim} \limits_{t \to \infty} (log_t(V_t)) \cup \mathop{\lim} \limits_{t \to \infty} (log_t(W_t))$
	\end{prop}
	
	\begin{proof}
	It is clear that $\mathop{\lim} \limits_{t \to \infty} (log_t(V_t)) \cup \mathop{\lim} \limits_{t \to \infty} (log_t(W_t)) \subset \mathop{\lim} \limits_{t \to \infty} (log_t(V_t\cup W_t))$ \\
	Conversely, say $P\in \mathop{\lim} \limits_{t \to \infty} (log_t(V_t\cup W_t))$. Then there exists a sequence $\{a_k\} \subset V_{t_k}\cup W_{t_k}$ such that $\mathop{\lim} \limits a_k = P$. But $\{a_k\}$ contains either infinitely many points from $V_t$ or from $W_t$. Thus $P\in \mathop{\lim} \limits_{t \to \infty} (log_t(V_t)) \cup \mathop{\lim} \limits_{t \to \infty} (log_t(W_t))$
	\end{proof}
	
      Tropicalization gives the opportunity to see some features of a given complex plane curve using a simpler picture in $\mathbb{R}^2$. For example, Mikhalkin \cite{mik2} found a simpler way of counting curves in $\mathbb{P}^2$ satisfying certain conditions by using tropical geometry. Recently many classical concepts in algebraic geometry have been translated into tropical geometry \cite{ardila} \cite{bert} \cite{gath} \cite{gath2} \cite{gath3} \cite{iten} \cite{shu}.

\newpage

\subsection{Tropical Nets}
	
	\subsubsection{Tropicalization of a $(3,2)$-net:}
			
	As an example, we find a tropicalization of the following (3,2)-net. In this case, it is possible to find a tropicalization in which all line families have distinct tropical limits. 
	
	\begin{center} $\ell_{11}=\{x=0\}\ \ \ell_{12}=\{y-z=0\} \ \ \ell_{21}=\{z=0\} $ $\ell_{22}=\{x-y=0\} \ \ \ell_{31}=\{y=0\} \ \ \ell_{32}=\{x-z=0\}$ 
	\end{center}
			
	We will denote each line in $\mathbb{P}^2$ by its dual coordinates in $(\mathbb{P}^2)^*$. So  $ax+by+cz=0$ will be denoted by $[a : b : c]$ (or its transpose). We form the following matrix by writing the dual coordinates of  $\ell_{11}, \ell_{12}, \ell_{21}, \ell_{22}, \ell_{31}$ and $\ell_{32}$ in columns.
	
			$$
				\left.
					\begin{matrix}
						& \ell_{11} & \ell_{12} & \ell_{21} & \ell_{22} & \ell_{31} & \ell_{32} \\		
						\left[
						\begin{matrix} \\ \\ \\ \end{matrix} \right. &
						\begin{matrix} 1 \\ 0 \\ 0 \end{matrix}&
						\begin{matrix} 0 \\ 1 \\ -1 \end{matrix} &
						\begin{matrix} 0 \\ 0 \\ 1 \end{matrix}&
						\begin{matrix} 1 \\ -1 \\ 0 \end{matrix} &
						\begin{matrix} 0 \\ 1 \\ 0 \end{matrix}&
						\begin{matrix} 1 \\ 0 \\ -1 \end{matrix} &
						\left.
						\begin{matrix} \\ \\ \\ \end{matrix} \right] &		
					\end{matrix}
				\right.
			$$
			
			We will apply a linear transformation with coefficients in $\mathbb{C}[t]$ to this configuration. The logarithmic limit of this family will give the tropicalization of this net.

$$\left[
            \begin{array}{ccc}
              t & t^2 & t^4 \\
              t^3 & t & t^2 \\
              t^2 & t^5 & 1 \\
            \end{array}
        \right]
        \left[
          \begin{array}{cccccc}
            1 & 0 & 0 & 1 & 0 & 1 \\
            0 & 1 & 0 & -1 & 1 & 0 \\
            0 & -1 & 1 & 0 & 0 & -1 \\
          \end{array}
        \right]
        $$
        
        Looking at the $z\neq 0$ chart, the lines transform to the following lines after this tropicalization:
        
        \begin{equation*} 
        \begin{split} 
        \ell_{11}&: (t)x+(t^3)y+(t^2)=0 \\
        \ell_{12}&: (t^2-t^4)x+(t-t^2)y+(t^5-1)=0\\
        \ell_{21}&: (t^4)x+(t^2)y+1=0\\
        \ell_{22}&: (t-t^2)x+(t^3-t)y+(t^2-t^5)=0\\
        \ell_{31}&: (t^2)x+(t)y+(t^5)=0\\
        \ell_{32}&: (t-t^4)x+(t^3-t^2)y+(t^2-1)=0
        \end{split}
        \end{equation*}
        
        Now we want to determine the centers of the resulting tropical lines. For simplicity we write the highest powers of $t$ in the coefficients of $x,y,z$ in a matrix.
        
			$$
				\left.
					\begin{matrix}
						& \L_{11} & \L_{12} & \L_{21} & \L_{22} & \L_{31} & \L_{32} \\		
						\left[
						\begin{matrix} \\ \\ \\ \end{matrix} \right. &
						\begin{matrix} 1 \\ 3 \\ 2 \end{matrix}&
						\begin{matrix} 4 \\ 2 \\ 5 \end{matrix} &
						\begin{matrix} 4 \\ 2 \\ 0 \end{matrix}&
						\begin{matrix} 2 \\ 3 \\ 5 \end{matrix} &
						\begin{matrix} 2 \\ 1 \\ 5 \end{matrix}&
						\begin{matrix} 4 \\ 3 \\ 2 \end{matrix} &
						\left.
						\begin{matrix} \\ \\ \\ \end{matrix} \right] &		
					\end{matrix}
				\right.
			$$
     
			We use \textit{Proposition 2.9} to find the centers of the lines. We subtract the first and second row from the third. The numbers in the first row(the highest power of $t$ as a coefficient of $x$) shift the center in the negative $x$-direction, the numbers in the second row shift the graph in the negative $y$-direction and the numbers in the third row shift the graph to the positive $x$-direction and the positive $y$-direction by the same amount.
			
			$$
				\left.
					\begin{matrix}
						& \L_{11} & \L_{12} & \L_{21} & \L_{22} & \L_{31} & \L_{32} \\		
						\left[
						\begin{matrix} \\ \\ \end{matrix} \right. &
						\begin{matrix} 1 \\ -1 \end{matrix}&
						\begin{matrix} 1 \\ 3 \end{matrix} &
						\begin{matrix} -4 \\ -2 \end{matrix}&
						\begin{matrix} 3 \\ 2 \end{matrix} &
						\begin{matrix} 3 \\ 4 \end{matrix}&
						\begin{matrix} -2 \\ -1 \end{matrix} &
						\left.
						\begin{matrix} \\ \\ \end{matrix} \right] &		
					\end{matrix}
				\right.
			$$ 
			
        After this procedure, the graph looks like the one in \textit{Figure 6}.
	\begin{figure}[h] \begin{center} \resizebox{8cm}{6cm}{\includegraphics{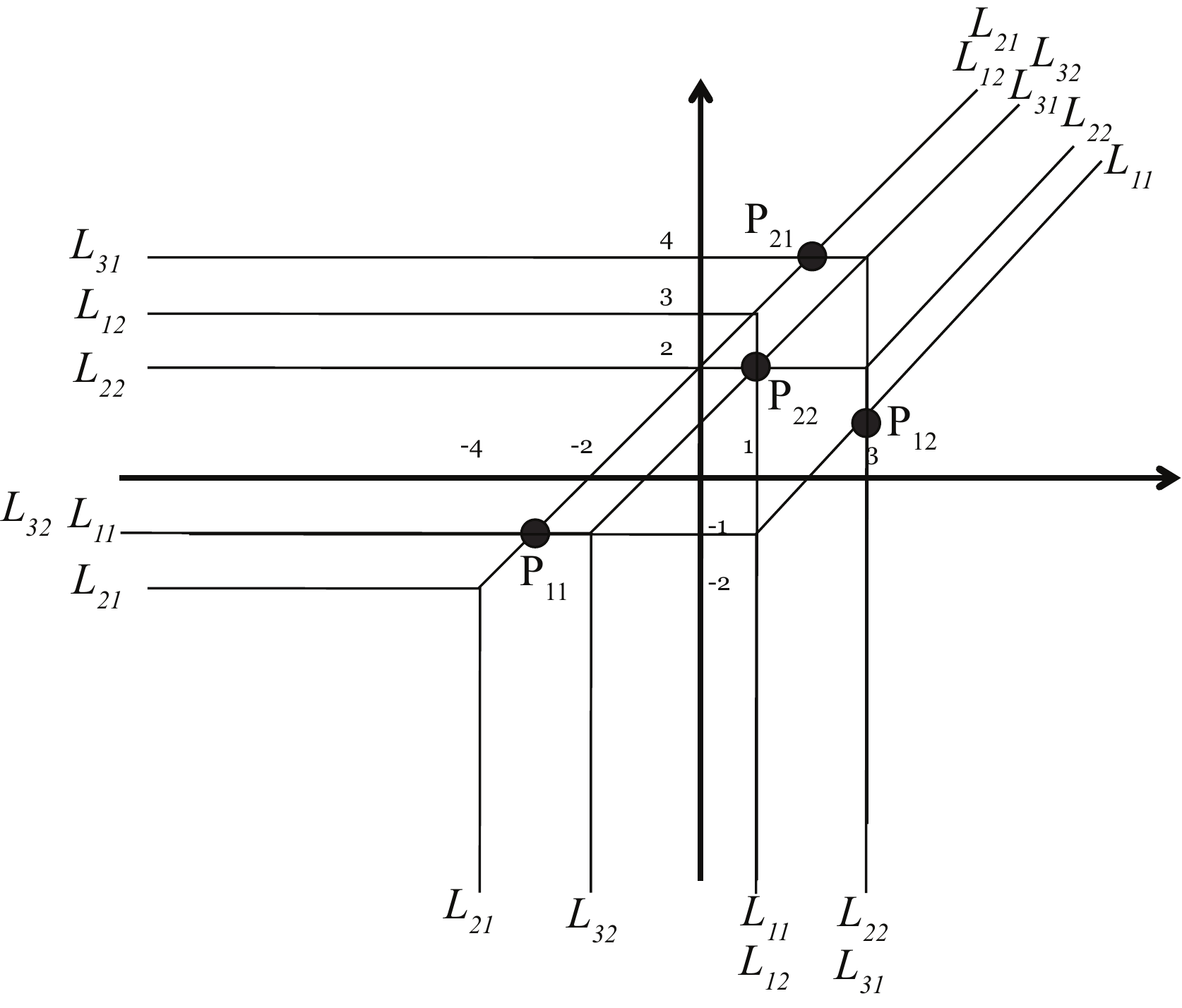}} \caption{A (3,2) Tropical Net} \label{fig:tropicalnet} \end{center}\end{figure}

\section{Nonexistence Of (4,4)-Nets}

\subsection{Orthogonal Latin Squares Of Order 4}
	In the next section we prove the nonexistence of (4,4)-nets. We need two orthogonal Latin squares (OLS) of order 4 to construct an abstract (4,4)-net.
	
	\begin{prop} The following is the unique pair of orthogonal Latin squares (OLS) of order 4 up to relabeling the numbers, and reordering rows and columns.
	
	\begin{center}
	$\left \{ \left[
            \begin{array}{cccc}
              1 & 2 & 3 & 4 \\
              2 & 1 & 4 & 3 \\
              3 & 4 & 1 & 2 \\
              4 & 3 & 2 & 1 \\
            \end{array}
          \right] \right.$
          $,
          \left. \left[
           \begin{array}{cccc}
              1 & 2 & 3 & 4 \\
              3 & 4 & 1 & 2 \\
              4 & 3 & 2 & 1\\
              2 & 1 & 4 & 3\\
            \end{array}
          \right] \right \} $
	\end{center}
	\end{prop}
	
	\begin{proof}
	Without loss of generality, we may assume 
		\begin{center}
	M=$\left[
            \begin{array}{cccc}
              1 & 2 & 3 & 4 \\
              2 &  &  &  \\
              3 &  &  &  \\
              4 &  &  &  \\
            \end{array}
          \right]$ and N=
          $
        \left[
           \begin{array}{cccc}
              1 & 2 & 3 & 4 \\
               &  &  &  \\
               &  &  & \\
               &  &  & \\
            \end{array}
          \right] $
	\end{center}
	
	If $N_{21}=3\Rightarrow N_{31}=4$ and $N_{41}=2$.\\
	If $N_{21}=4$, then change the roles of M and N, and reorder the rows. We are back to the $N_{21}=3$ case. So we have 
		\begin{center}
		M=$\left[
            \begin{array}{cccc}
              1 & 2 & 3 & 4 \\
              2 &  &  &  \\
              3 &  &  &  \\
              4 &  &  &  \\
            \end{array}
          \right]$ and N=
          $
        \left[
           \begin{array}{cccc}
              1 & 2 & 3 & 4 \\
              3 &  &  &  \\
              4 &  &  & \\
              2 &  &  & \\
            \end{array}
          \right] $
	\end{center}
	
	Now, either $M_{22}=M_{44}=3$ or $M_{24}=M_{42}=3$.\\
	But if $M_{22}=M_{44}=3$, the (3,2) pair requires $N_{22}=2$ or $N_{44}=2$, contradiction.\\
	So $M_{24}=M_{42}=3 \Rightarrow N_{24}=2$ and $N_{42}=1$.
	
	\begin{center}
		M=$\left[
            \begin{array}{cccc}
              1 & 2 & 3 & 4 \\
              2 &  &  & 3 \\
              3 &  &  &  \\
              4 & 3 &  &  \\
            \end{array}
          \right]$ and N=
          $
        \left[
           \begin{array}{cccc}
              1 & 2 & 3 & 4 \\
              3 &  &  & 2 \\
              4 &  &  & \\
              2 & 1 &  & \\
            \end{array}
          \right] $
	\end{center}
	
	We immediately get $N_{44}=3,\ N_{34}=1,\ N_{43}=4,\ N_{23}=1,\ N_{33}=2,\ N_{32}=3,\ N_{22}=4$
	
	\begin{center}
		M=$\left[
            \begin{array}{cccc}
              1 & 2 & 3 & 4 \\
              2 &  &  & 3 \\
              3 &  &  &  \\
              4 & 3 &  &  \\
            \end{array}
          \right]$ and N=
          $
        \left[
           \begin{array}{cccc}
              1 & 2 & 3 & 4 \\
              3 & 4 & 1 & 2 \\
              4 & 3 & 2 & 1\\
              2 & 1 & 4 & 3\\
            \end{array}
          \right] $
	\end{center}
	
	In order to complete the remaining entries of M, look at (1,2). It can only occur at ($M_{33},N_{33}$), so $M_{33}=1$.\\
	So $M_{23}=4,\ M_{22}=1,\ M_{32}=4,\ M_{34}=2,\ M_{43}=2,\ M_{44}=1$
	
	\begin{center}
		M=$\left[
            \begin{array}{cccc}
              1 & 2 & 3 & 4 \\
              2 & 1 & 4 & 3 \\
              3 & 4 & 1 & 2 \\
              4 & 3 & 2 & 1 \\
            \end{array} ,    
          \right]$ and N=
          $
        \left[
           \begin{array}{cccc}
              1 & 2 & 3 & 4 \\
              3 & 4 & 1 & 2 \\
              4 & 3 & 2 & 1\\
              2 & 1 & 4 & 3\\
            \end{array}
          \right] $
	\end{center}
	
	So the proof is complete.
	\end{proof}

	This pair of OLS of order 4 gives us an abstract (4,4)-net. 
\newpage

\subsection{The Incidence Structure Of The Possible (4,4)-Net}	
		
		Suppose that we have a hypothetical (4,4)-net ($\mathcal{A}_1,\mathcal{A}_2,\mathcal{A}_3,\mathcal{A}_4,\mathcal{X}$) in $\mathbb{C}P^2$. Denote the sets of its lines by 	
		$$
		\mathcal{A}_k = \{ \ell_{k1}, \ldots, \ell_{k4} \} \ where \ k \in \{1,\ldots, 4\}
		$$
		and $\mathcal{X}= \{p_{ij} \} \ where \ i,j\in \{1\ldots 4 \} $ 

		and the points are labeled as $p_{ij}=\ell_{1i}\cap \ell_{2j}$.

		Then by regarding the OLS of order 4 we find the incidence relations,		
		$$ p_{ij} = \ell_{1i}\cap \ell_{2j} \cap \ell_{3M_{ij}} \cap \ell_{4N_{ij}} $$					\subsection{A Tropicalization Of The Possible (4,4)-Net}	

By using the fundamental theorem of projective geometry we can find a unique transformation between the lines $\ell_{11}, \ell_{12}, \ell_{21}$ and $\ell_{22}$ and $z=0$, $x+y+z=0$, $x=0$ and $y=0$ respectively. Note that no 3 of $\ell_{11},\ell_{12},\ell_{21}$ and $\ell_{22}$ are concurrent because of the net axioms.

	Now we will find the new location of the points $p_{11}, p_{12}, p_{21}$ and $p_{22}$ after the transformation. 
	
	\begin{eqnarray*}
	p_{11} &=& \ell_{11} \cap \ell_{21} = (0:1:0)  \\
	p_{12} &=& \ell_{11} \cap \ell_{22} = (1:0:0)  \\ 
	p_{21} &=& \ell_{12} \cap \ell_{21} = (0:1:-1)   \\
	p_{22} &=& \ell_{12} \cap \ell_{22} = (1:0:-1)  
	\end{eqnarray*}
	
The incidence relations immediately give two more pieces of information, that is, the equations of the lines $\ell_{31}$ and $\ell_{32}$. Since $\ell_{31}$ passes through the points $p_{11}$ and $p_{22}$, the equation of $\ell_{31}$ is $x+z=0$. Similarly $\ell_{32}$ passes through $p_{12}$ and $p_{21}$, therefore its equation is $y+z=0$.

	\subsubsection{The Tropicalization Of The Lines}	
	Writing the dual coordinates of $\ell_{11}, \ell_{12}, \ell_{21}, \ell_{22}, \ell_{31}$ and $\ell_{32}$ in columns, we get the following matrix:
			$$
				\left.
					\begin{matrix}
						& \ell_{11} & \ell_{12} & \ell_{21} & \ell_{22} & \ell_{31} & \ell_{32} \\		
						\left[
						\begin{matrix} \\ \\ \\ \end{matrix} \right. &
						\begin{matrix} 0 \\ 0 \\ 1 \end{matrix}&
						\begin{matrix} 1 \\ 1 \\ 1 \end{matrix} &
						\begin{matrix} 1 \\ 0 \\ 0 \end{matrix}&
						\begin{matrix} 0 \\ 1 \\ 0 \end{matrix} &
						\begin{matrix} 1 \\ 0 \\ 1 \end{matrix}&
						\begin{matrix} 0 \\ 1 \\ 1 \end{matrix} &
						\left.
						\begin{matrix} \\ \\ \\ \end{matrix} \right] &			
					\end{matrix}
				\right.
			$$

Tropicalize the net as explained in \textit{Section 2.3} using the matrix
	$$
           T=\left[
             \begin{array}{ccc}
              t-t^2 & t^2-t^4 & t^4 \\
              t^3+t^2 & 1 & -t^3 \\
              t^2 & t^5 & 1\\
            \end{array}
          \right] $$
          
\begin{prop}
The matrix above has nonzero determinant except for finitely many values of t.
\end{prop}

\begin{proof}
The determinant of $T$ is $\det(T)=t-t^2-t^4-t^5+2t^9-t^{10}+t^{11}+t^{12}$. $T=0$ has 4 real and 8 complex roots. Hence the determinant is nonzero except for these values of \textit{t}.
\end{proof}

This tropicalization sends the lines $\ell_{11}, \ell_{12}, \ell_{21}, \ell_{22}, \ell_{31}$ and $\ell_{32}$ to the tropical lines $\L_{11}, \L_{12}, \L_{21}, \L_{22}, \L_{31}$ and  $\L_{32}$ respectively, whose centers are listed below in columns.
			$$
				\left.
					\begin{matrix}
						& \L_{11} & \L_{12} & \L_{21} & \L_{22} & \L_{31} & \L_{32} \\		
						\left[
						\begin{matrix} \\ \\ \end{matrix} \right. &
						\begin{matrix} -4 \\ -3 \end{matrix}&
						\begin{matrix} 4 \\ 3 \end{matrix} &
						\begin{matrix} 0 \\ -1 \end{matrix}&
						\begin{matrix} 1 \\ 5 \end{matrix} &
						\begin{matrix} -2 \\ 0 \end{matrix}&
						\begin{matrix} 3 \\ 2 \end{matrix} &
						\left.
						\begin{matrix} \\ \\ \end{matrix} \right] &		
					\end{matrix}
				\right.
			$$

\begin{lem} Let $ax+by+cz=0$ be a line in $\mathbb{P}^2$ and $\ell_{ij}=[a:b:c]\in (\mathbb{P}^2)^*$ We use the transformation 
	          $$
           \left[
             \begin{array}{ccc}
              t-t^2 & t^2-t^4 & t^4 \\
              t^3+t^2 & 1 & -t^3 \\
              t^2 & t^5 & 1\\
            \end{array}
          \right] $$
	Say $\psi(\ell_{ij})$ denotes the coordinates of the center of $\L_{ij}$. Then
	\renewcommand{\labelenumi}{\roman{enumi}.}
	\renewcommand{\labelenumii}{\roman{.}.}
	\begin{enumerate}
		\item 
			\begin {itemize}
			\item $\psi([0:0:1])=(-4,-3)$
			\item $\psi([0:1:0])=(1,5)$
			\item $\psi([1:0:0])=(0,-1)$
			\item $\psi([1:1:1])=(4,3)$
			\item $\psi([1:0:1])=(-2,0)$
			\end{itemize}
		\item
			\begin {itemize}
			\item $\psi([a:1:1])=(3,2)\ for\ a\neq 1$
			\item $\psi([a:0:1])=(-2,-1) \ for \ a\notin \{0,1\}$
			\item $\psi([1:b:1])=(1,3) \ for \ b\notin \{0,1\} $
			\end{itemize}
          	\item For all other values of $[a:b:c]\ ,\ \psi([a:b:c])=(1,2) $
        \end{enumerate}
        In particular for the values of $(x,y)$ in $Part\ (i)$, $\psi^{-1}(x,y)$ contains a unique line.
\end{lem}

\begin{proof} 
	          $
           \left[
             \begin{array}{ccc}
              t-t^2 & t^2-t^4 & t^4 \\
              t^3+t^2 & 1 & -t^3 \\
              t^2 & t^5 & 1\\
            \end{array}
          \right] 
          $ 
       $
                   \left[
            \begin{array}{c}
             a \\
             b \\
             c\\
           \end{array}
         \right]
          $
          =
            $
                     \left[
             \begin{array}{c}
             a(t-t^2) + b(t^2-t^4) + ct^4 \\
              a(t^3+t^2)+ b-ct^3 \\
              at^2 + bt^5 + c\\
            \end{array}
          \right]
          $
          \\ \\ 
      If we replace the values for $a,\ b$ and $c$, we get the centers above.
\end{proof}

The centers of $L_{11}, \L_{12}, \L_{21}, \L_{22}, \L_{31}$ and $\L_{32}$ were given before \textit{Lemma 3.1}. The different possible centers for the lines other than $\L_{11}, \L_{12}, \L_{21}, \L_{22}, \L_{31}$ and $\L_{32}$ are (3,2), (-2,-1), (1,3), (1,2). The graphs of the lines $\L_{11}, \L_{12}, \L_{21}, \L_{22}, \L_{31}$ and $\L_{32}$ are given in \textit{Figure 7}.

\begin{figure} \begin{center} \resizebox{8cm}{6cm} {\includegraphics{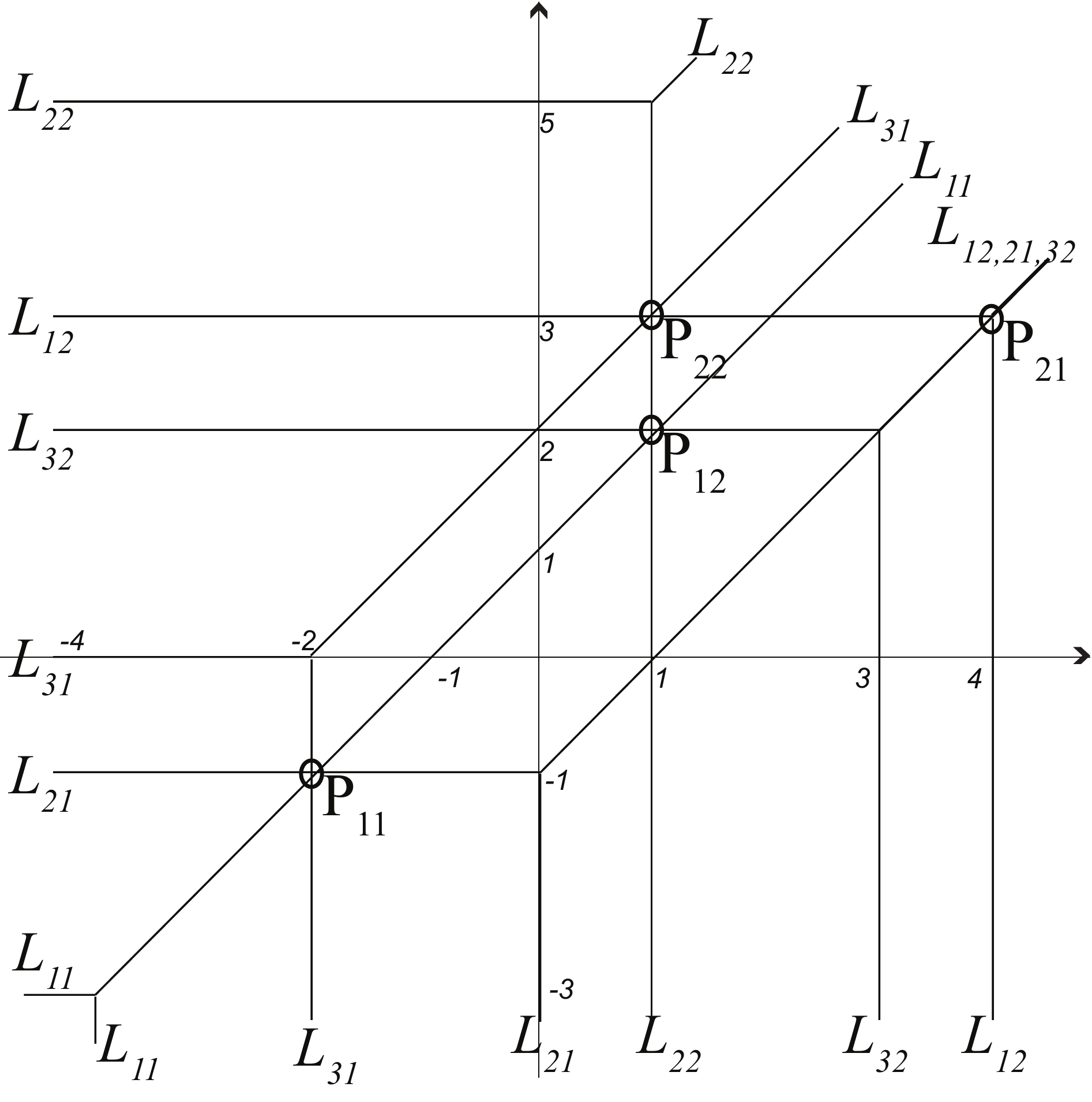}} \caption{The lines and points in the tropical plane after tropicalization} \label{fig:fourlines} \end{center}\end{figure}

\newpage

\subsection{The Tropicalization Of The Points}	

\begin{lem}  If a matrix $M(t)$ is used to tropicalize the coordinates of the lines in $(\mathbb{P}^2)^*$, then adj$(M(t))$ can be used to tropicalize the coordinates of the points in $\mathbb{P}^2$.
\end{lem}

\begin{proof} 
	Let $p$ be the coordinates in $\mathbb{P}^2$ and $\ell$ be the coordinates in $(\mathbb{P}^2)^*$ as column vectors. Let  $\ell_t =M(t)\ell$, and $p_t$ the coordinates of the point after the relevant transformation. We know that $p^T\ell = 0$ $\Leftrightarrow$ $p$ lies on $\ell$ $\Leftrightarrow$ $p_t$ lies on $\ell_t$ $\Leftrightarrow p_t^T\ell_t= 0 $ and $\ell_t = M\ell$. \\ 
We claim that $p_t^T = p^TM^{-1}$. This is because
	$$
		p_t^T \ell_t = p^TM^{-1}M\ell = p^T\ell = 0 \\
	$$
	Multiplying both sides of $p_t^T=p^TM^{-1}$ by $\det(M)$ does not change the homogenous coordinates. Therefore 
	$$
		p_t^T = p^TM^{-1}(\det{M}) = p^T(\mathrm{adj}M)^T 
	$$
	where $(\mathrm{adj}M)^T$ is the transpose of adj$M$. 
		          $$ \Rightarrow p_t =  (\mathrm{adj}M)p $$
\end{proof}
 \newpage

Now let

	$$ M=
           \left[
             \begin{array}{ccc}
              t-t^2 & t^2-t^4 & t^4 \\
              t^3+t^2 &1 & -t^3 \\
              t^2 & t^5 & 1\\
            \end{array}
          \right] $$

	Then
          $$ \mathrm{adj}(M)=
           \left[
             \begin{array}{ccc}
             	\left| \begin{array}{cc} 1 & -t^3 \\ t^5 &1 \\ \end{array} \right|  & 
            	-\left|  \begin{array}{cc} t^2+t^3 & -t^3 \\  t^2 & 1 \\ \end{array} \right|  & 
            	\left| \begin{array}{cc} t^2+t^3 & 1 \\ t^2 & t^5 \\ \end{array} \right| \\
              
              	- \left| \begin{array}{cc} t^2-t^4 & t^4 \\ t^5 &1 \\ \end{array}\right|  &               
		 \left| \begin{array}{cc} t-t^2 & t^4 \\ t^2 & 1 \\ \end{array} \right|  & 
		- \left| \begin{array}{cc} t-t^2 & t^2-t^4 \\ t^2 & t^5 \\ \end{array} \right| \\

             	\left| \begin{array}{cc} t^2-t^4 & t^4 \\ 1 & -t^3 \\ \end{array} \right|  & 
            	-\left|  \begin{array}{cc} t-t^2 & t^4 \\  t^2+t^3 & -t^3 \\ \end{array} \right|  & 
            	\left| \begin{array}{cc} t-t^2 & t^2-t^4 \\ t^2+t^3 & 1 \\ \end{array} \right| \\
	
            \end{array}
          \right] $$
          
                    $$ = 
           \left[
             \begin{array}{ccc}
              1+t^8 & -t^2-t^3-t^5 & -t^2+t^7+t^8 \\
              -t^2+t^4+t^9 & t-t^2-t^6 & t^4-2t^6+t^7 \\
              -t^4-t^5+t^7 & t^4-t^5+t^6+t^7 & t-t^2-t^4-t^5+t^6+t^7 \\
            \end{array}
          \right] $$	
          
\begin{rem} In the process of dehomogenization the effect of the transformation on the points is the reverse of the effect of the transformation on the dual coordinates. For the points we subtract the highest power of the third row from the highest power of the first row and the highest power of the second row. The first and the second numbers shall be the first and second coordinates of the location of the point.
\end{rem}


\begin{lem} Let $p_{ij}=(a:b:c) \in \mathbb{P}^2$. We use the transformation 
	$$
           \left[
             \begin{array}{ccc}
              1+t^8 & -t^2-t^3-t^5 & -t^2+t^7+t^8 \\
              -t^2+t^4+t^9 & t-t^2-t^6 & t^4-2t^6+t^7 \\
              -t^4-t^5+t^7 & t^4-t^5+t^6+t^7 & t-t^2-t^4-t^5+t^6+t^7 \\
            \end{array}
          \right]
          $$
          
          Say $\varphi (p_{ij})=P_{ij}$. Then 
          
          \renewcommand{\labelenumi}{\roman{enumi}.}
	\renewcommand{\labelenumii}{\roman{.}.}
	\begin{enumerate}
		\item 
			\begin {itemize}
			\item $\varphi((0:1:0))=(-2,-1)$
			\item $\varphi((1:0:-1))=(1,3)$
			\item $\varphi((0:1:-1))=(4,3)$
			\end{itemize}
		\item
			\begin {itemize}
			\item $\varphi((1:b:-1-b))=(2,3)\ for\ b\neq 0$
			\item $\varphi((0:b:1))=(1,0) \ for \ b\neq -1$
			\item $\varphi((1:b:-1))=(0,2) \ for \ b\neq 0 $
			\end{itemize}
          	\item For all other values of $(a:b:c)\ ,\ \varphi((a:b:c))=(1,2) $
        \end{enumerate}
        
         In particular each of $\varphi^{-1}(-2,-1),\ \varphi^{-1}(1,3),\ \varphi^{-1}(4,3)$ is a single point.

\end{lem}

\begin{proof}
This follows from the same argument as in the proof of \textit{Lemma 3.1}
\end{proof}

\subsubsection{Point line table}	

We collect below the information about the points on the tropical plane which can either be the center of a point or a tropical line after the degeneration. The lines column gives the classical equation of the line if the coordinate is a center of a line of the net. The points column gives the location of the classical point if there is a point of the net on that coordinate. If there is no a point or line there we write NS as an abbreviation of Not Special. Note that we show the coordinate of a point by $(a:b:c)$ and of a line by $[d:e:f]$. We write all the relations between $a,b,c$ and $d,e,f$ in the table.
			          			
	$$
		\left.
			\begin{matrix}
				& coordinates & points &lines & \\
				& & (a:b:c) & [d:e:f] & \\		
				\left[
				\begin{matrix} \\ \\ \\ \\ \\ \\ \\ \\ \\ \\ \\ \\ \\ \\ \\ \\ \\ \\ \\ \\ \\ \\ \\ \\ \\ \\ \end{matrix} \right. &
				\begin{matrix} (1,2) \\  \\ (0,2) \\ \\ \\ (2,3)\\ \\  (1,0) \\ \\ (4,3) \\ \\  (1,3) \\ \\ (-2,-1)  \\ \\ \\ (3,2) \\ \\ \\  (-2,0) \\ \\ (0,-1) \\ \\ (1,5) \\  \\ (-4,-3) \\ \end{matrix} &
				\begin{matrix}\\ no \ relations \\ \\ a+c=0 \\ \\ \{a+b+c=0 \} \ or\\ \{ c=0, a+b=0\} \\ \\ \{a=0\} \ or\ \{ a,b=0\} \\ \\ a=0, b+c=0\\ \\ b =0, a+c=0\\ \\ a=0,c=0\\ \\ \\ NS \\ \\ \\ NS \\ \\ NS \\ \\ NS \\ \\ NS \\ \\ \end{matrix} &
				\begin{matrix}\\ \\ no \ relations \\ \\ NS \\ \\  \\ NS \\ \\ NS \\ \\ d=e=f \\ \\ d=f \\ \\ e=0 \\ \\ \\ \{e=f\} \ or\\ \{d=0,e=f\} \\ \\ e=0, d=f \\ \\ e=0, f=0 \\ \\ d=0, f=0 \\ \\ d=0,e=0 \\ \\ \\ \end{matrix}&

				\left.
				\begin{matrix} \\ \\ \\ \\ \\ \\ \\ \\ \\ \\ \\ \\ \\ \\ \\ \\ \\ \\ \\ \\ \\ \\ \\ \\ \\ \\ \end{matrix} \right] &					\end{matrix}
				\right.
	$$ 

\subsubsection{Determining Other Lines And Points}

Up to now we know the equations of $\L_{11}, \L_{12}, \L_{21}, \L_{22}, \L_{31}, \L_{32} $ and the location of the points $P_{11}, P_{12}, P_{21}, P_{22}$. Now we will determine some of the other lines and points by using the tables above. 

By using the incidence relations, $\L_{31}$ passes through $P_{11}, P_{22}, P_{33}, P_{44}$. Let us determine the locations of $P_{33}$ and $P_{44}$. The line $\L_{31}$ has its center at $(-2,0)$. The only location of the points on $\L_{31}$ may be $(-2,-1), (0,2), (1,3)$. Now, we check the point line table to look at the classical coordinates. $(-2,-1)$ belongs to $(0:1:0)$ which is $p_{11}$. The coordinate $(1,3)$ belongs to $(1:0:-1)$ which is $p_{22}$. The coordinate $(0,2)$ belongs to $(1:b:-1)$ where $b\in \mathbb{Z}-\{0\}$. Therefore $p_{33}=(1:m_1:-1)$ and $p_{44}=(1:m_2:-1)$ where $m_1,m_2 \in \mathbb{Z}-\{0\}$ $m_1\neq m_2$

Similarly considering $P_{11}, P_{21},  P_{31}, P_{41}$ on $\L_{21}$, we get  $p_{31}=(0:1:t_1)$ and $p_{41}=(0:1:t_2)$ where $t_1,t_2 \in \mathbb{Z}-\{-1,0\}$ $t_1\neq t_2$

By the same methods, looking at $P_{21}, P_{22},  P_{23}, P_{24}$ on $\L_{12}$, we get $p_{23}=(1:s_1:-1-s_1)$ and $p_{24}=(1:s_2:-1-s_2)$ where $s_1,s_2 \in \mathbb{Z}-\{0\}$ $s_1\neq s_2$ 

Now we want to determine the center of the line $\L_{41}$. The line $\L_{41}$ passes through $P_{11}$. The only possible centers for a line passing through $P_{11}$ are $(-2,-1), (-2,0), (0,-1)$ $(-4,-3)$. Considering the point-line table, if a line has a center at $(-2,0)$ then its equation is $x+z=0$, so the line is $\ell_{31}$. Similarly if a line has a center at $(0,-1)$ then its equation is $x=0$, so the line is $\ell_{21}$. If a line has a center at $(-4,-3)$ then its equation is $z=0$, so the line is $\ell_{11}$. Therefore the only possible center for $\L_{41}$ is $(-2,-1)$ and the equation of $\ell_{41}$ is $x+k_1z=0$ where $k_1\in \mathbb{Z}-\{0,1\}$.

Similarly considering $p_{21}$ on $\L_{43}$, we get that the center of $\L_{43}$ is $(3,2)$ and the equation of $\ell_{43}$ is $k_2x+y+z=0$ where $k_2 \in \mathbb{Z}-\{0,1\}$.

If we make similar calculations for $\L_{44}$, by using that $P_{22}$ lies on $\L_{44}$ we get that the center of $\L_{44}$ is  $(1,3)$ and the equation of $\ell_{44}$ is $x+k_3y+z=0$ where $k_3\in \mathbb{Z}-\{0,1\}$.

\newpage

Up to now we have the following data:

	\begin{eqnarray*}
		\ell_{11} &:& z = 0  \\
		\ell_{12} &:&  x+y+z=0 \\
		\ell_{21} &:& x=0 \\
		\ell_{22} &:& y=0 \\
		\ell_{31} &:& x+z = 0 \\
		\ell_{32} &:& y+z = 0 \\
		\ell_{41} &:& x+k_1z=0 \\
		\ell_{43} &:& k_2x+y+z=0 \\
		\ell_{44} &:& x+k_3y+z = 0 
	\end{eqnarray*}			
	\begin{eqnarray*}
		p_{11} &:& (0:1:0)  \\
		p_{12} &:&  (1:0:0) \\
		p_{21} &:& (0:1:-1) \\
		p_{22} &:& (1:0:-1) \\
		p_{23} &:& (1:s_1:-1-s_1) \\
		p_{24} &:& (1:s_2:-1-s_2) \\
		p_{31} &:& (0:1:t_1) \\
		p_{33} &:& (1:m_1:-1) \\
		p_{41} &:& (0:1:t_2) \\
		p_{44} &:& (1:m_2:-1) 
	\end{eqnarray*}
	
	\begin{eqnarray*}
		k_1,k_2,k_3\in \mathbb{Z}-\{0,1\} \\ 
		m_1,m_2 \in \mathbb{Z}-\{0\}\ and \ m_1\neq m_2 \\
		s_1,s_2 \in \mathbb{Z}-\{0\}\   and \ s_1\neq s_2 \\
		t_1,t_2 \in \mathbb{Z}-\{-1,0\}\  t_1\neq t_2
	\end{eqnarray*}

\newpage
We determine the other points by using the intersection relations.
		\begin{eqnarray*}
		p_{11} &=& \ell_{11}\cap \ell_{21} = [0\ 0\ 1] \cap [1\ 0\ 0] = (0:1:0) \\
		p_{12} &=& \ell_{11}\cap \ell_{22} = [0\ 0\ 1] \cap [0\ 1\ 0] = (1:0:0) \\
		p_{13} &=& \ell_{11}\cap \ell_{43} = [0\ 0\ 1] \cap [k_2\ 1\ 1] = (1:-k_2:0) \\
		p_{14} &=& \ell_{11}\cap \ell_{44} = [0\ 0\ 1] \cap [1\ k_3\ 1] = (-k_3:1:0) \\
		p_{21} &=& \ell_{12}\cap \ell_{21} = [1\ 1\ 1] \cap [1\ 0\ 0] = (0:1:-1) \\
		p_{22} &=& \ell_{12}\cap \ell_{22} = [1\ 1\ 1] \cap [0\ 1\ 0] = (1:0:-1) \\
		p_{23} &=& \ell_{12}\cap \ell_{41} = [1\ 1\ 1] \cap [1\ 0\ k_1] = (-k_1:k_1-1:1) \\
		p_{24} &=& (1:s_2:-1-s_2) \\
		p_{31} &=& \ell_{21}\cap \ell_{44} = [1\ 0\ 0] \cap [1\ k_3\ 1] = (0:1:-k_3) \\
		p_{32} &=& \ell_{22}\cap \ell_{43} = [0\ 1\ 0] \cap [k_2\ 1\ 1] = (1:0:-k_2) \\
		p_{33} &=& (1:m_1:-1) \\
		p_{34} &=& \ell_{32}\cap \ell_{41} = [0\ 1\ 1] \cap [1\ 0\ k_1] = (k_1:1:-1) \\
		p_{41} &=&  (0:1:t_2)\\
		p_{42} &=& \ell_{22}\cap \ell_{41} = [0\ 1\ 0] \cap [1\ 0\ k_1] = (-k_1:0:1) \\
		p_{43} &=& \ell_{32} \cap \ell_{44} = [0\ 1\ 1] \cap [1\ k_3\ 1] = (k_3-1:-1:1) \\
		p_{44} &=& \ell_{31}\cap \ell_{43} = [1\ 0\ 1] \cap [k_2\ 1\ 1] = (-1:k_2-1:1) \\
	\end{eqnarray*}	
	
	We know the equations of the 9 lines up to here. We determine some of the other lines by using the points above. We need two points to determine a line. By using the extra points we get some new equations.

$\\ \underline{\ell_{13}} : $
	\begin{eqnarray}
		Let \ \ell_ {13} = [x: y: z] \in (\mathbb{P}^2)^*  \nonumber \\
		p_{31}, p_{32}, p_{33}, p_{34} \in \ell_{13} \nonumber \\
		y-k_3z=0 \\
		x-k_2z=0 \\
		x+m_1y-z=0 \\
		k_1x+y-z=0 \\
		By \ (3.1), \ if \ z=1 \Rightarrow y=k_3,\ by\  (3.2), \  x=k_2 \nonumber \\
		So\ \ell_{13}=[k_2: k_3: 1 ] \nonumber \\
		By \ (3.3) \ k_2+m_1k_3-1=0 \\
		By \ (3.4) \ k_1k_2+k_3-1=0 \\
		By \ (3.6)\  k_3 = 1-k_1k_2 
	\end{eqnarray}	

$\\ \underline{\ell_{14}} : $
	\begin{eqnarray}
		Let \ \ell_ {14} = [x: y: z] \in (\mathbb{P}^2)^*  \nonumber \\
		p_{41}, p_{42}, p_{43}, p_{44} \in \ell_{14} \nonumber \\
		y+t_2z=0 \\
		-k_1x+z=0 \\
		(k_3-1)x-y+z=0 \\
		-x+(k_2-1)y+z=0 \\
		By \ (3.9), \ if \ x=1 \Rightarrow z=k_1,\ by\  (3.8), \  y=-t_2k_1 \nonumber \\
		So\ \ell_{14}=[1: -t_2k_1: k_1 ] \nonumber \\
		By \ (3.10) \ (k_3-1)+t_2k_1+k_1=0 \\
		By \ (3.11) \ -1-(k_2-1)t_2k_1+k_1=0 
	\end{eqnarray}

$\\ \underline{\ell_{23}} : $
	\begin{eqnarray}
		Let \ \ell_ {23} = [x: y: z] \in (\mathbb{P}^2)^*  \nonumber \\
		p_{13}, p_{23}, p_{33}, p_{43} \in \ell_{23} \nonumber \\
		x-k_2y=0 \\
		-k_1x+(k_1-1)y+z=0 \\
		x+m_1y-z=0 \\
		(k_3-1)x-y+z=0 \\
		By \ (3.14), \ if \ y=1 \Rightarrow x=k_2 \nonumber \\
		By \ (3.15), \ z= k_1k_2-k_1+1  \nonumber\\
		So\ \ell_{23}=[k_2: 1: k_1k_2-k_1+1 ] \nonumber \\
		By \ (3.16) \ k_2+m_1-k_1k_2+k_1-1=0 \\
		By \ (3.17) \ k_1k_2+k_2k_3-k_1-k_2=0 
	\end{eqnarray}		
	
$\\ \underline{\ell_{34}} : $
	\begin{eqnarray}
		Let \ \ell_ {34} = [x:y:z] \in (\mathbb{P}^2)^*  \nonumber \\
		p_{14}, p_{23}, p_{32}, p_{41} \in \ell_{34} \nonumber \\
		-k_3x+y=0 \\
		-k_1x+(k_1-1)y+z=0 \\
		x-k_2z=0\\
		y+t_2z=0 \\
		By \ (3.22), \ if \ z=1 \Rightarrow x=k_2 \nonumber \\
		By \ (3.23), \ z= 1  \Rightarrow y=-t_2 \nonumber\\
		So\ \ell_{34}=[k_2: -t_2: 1 ] \nonumber \\
		By \ (3.21)\ k_1k_2+k_1t_2-t_2-1=0
	\end{eqnarray}	
	
$\\ \underline{\ell_{33}} : $
	\begin{eqnarray}
		Let \ \ell_ {33} = [x:y:z] \in (\mathbb{P}^2)^*  \nonumber \\
		p_{13}, p_{24}, p_{31}, p_{42} \in \ell_{33} \nonumber \\
		x-k_2y=0 \\
		x+s_2y+(-1-s_2)z=0 \\
		y-k_3z=0\\
		-k_1x+z=0 \\
		By \ (3.25), \ if \ y=1 \Rightarrow x=k_2 \nonumber \\
		By \ (3.28), \ z=k_1k_2 \nonumber\\
		So\ \ell_{33}=[k_2: 1: k_1k_2 ] \nonumber \\
		By \ (3.27)\ k_1k_2k_3=1
	\end{eqnarray}

	We get a contradiction as follows: 
	\begin{eqnarray}
		Using \ (3.7), \ (3.19) \ and \  k_1\neq 0 \ we\  get\  k_2^2-k_2+1=0 \\
		If\  we \ subtract\  (3.5)\  from \ (3.18)\ and\ using\ (3.7)\ we \ get\  m_1k_2-k_2+1=0 \\
		Using \ (3.7),\ (3.12) \ and \ k_1\neq 0\ we \ get \ t_2 = k_2-1 \\ 
		Using \ (3.13), \ (3.30) \ and\ (3.32)\ we \ get\ k_1+k_1k_2=1 \\
		Using\ (3.7)\ and\ (3.33) \ we \ get\ k_3=k_1 \\	
		Using \ (3.29) \ and\ (3.34)\ we \ get \ k_1^2k_2=1\ so\ that \ k_1k_2=\frac{1}{k_1} \\
		Using \ (3.24)\ and\ (3.32)\ we\ get\ 2k_1k_2-k_1-k_2=0 \\
		Using \ (3.33) \ and\ (3.36)\ we\ get\ k_2=2-3k_1 \\ 
		Using \ (3.37)\ and\ (3.30)\ we\ get\ 3k_1^2-3k_1+1=0 \\
		Using \ (3.33)\ and \ (3.35)\ we\ get\ k_1^2-k_1+1=0 \ which \ contradicts\  to \ (3.38) \nonumber
	\end{eqnarray}	
	
	This contradiction proves that there cannot be any (4,4)-nets in $\mathbb{C}P^2$.

\section{Uniqueness Of (4,3)-Net}

\subsection{Orthogonal Latin Squares Of Order 3}

Stipins proved in his thesis that there is a unique (4,3)-net \cite{stipins}. Here we give another proof by using tropical geometry. We need two orthogonal Latin squares of order 3 to construct an abstract (4,3)-net.
\begin{prop} The following is the unique pair of orthogonal Latin squares (OLS) of order 3 up to relabeling the numbers, and reordering rows and columns.
	
	\begin{center}
	$\left \{ \left[
            \begin{array}{ccc}
              1 & 2 & 3 \\
              2 & 3 & 1 \\
              3 & 1 & 2 \\
            \end{array}
          \right] \right.$
          $,
          \left. \left[
           \begin{array}{ccc}
              1 & 2 & 3 \\
              3 & 1 & 2 \\
              2 & 3 & 1 \\
            \end{array}
          \right] \right \} $
	\end{center}
	\end{prop}
	
	\begin{proof}
	Without loss of generality, we may assume 
		\begin{center}
	M=$\left[
            \begin{array}{ccc}
              1 & 2 & 3 \\
              2 &  & \\
              3 &  &  \\
            \end{array}
          \right]$ and N=
          $
        \left[
           \begin{array}{ccc}
              1 & 2 & 3 \\
               &  &  \\
               &  &  \\
            \end{array}
          \right] $
	\end{center}
	
	$M_{21}=2, \ N_{11}=1 \Rightarrow N_{21}=3$, $N_{22}=1$ and $N_{23}=2$.\\
	The other entries are straightforward.
\end{proof}

\newpage

\subsection{The Incidence Structure Of The Possible (4,3)-Net}	
		
		Suppose that we have a hypothetical (4,3)-net ($\mathcal{A}_1,\mathcal{A}_2,\mathcal{A}_3,\mathcal{X}$) in $\mathbb{C}P^2$. Denote the sets of its lines by 	
		$$
		\mathcal{A}_k = \{ \ell_{k1}, \ell_{k2}, \ell_{k3} \} \ where \ k \in \{1,2,3\}
		$$
		and $\mathcal{X}= \{p_{ij} \} \ where \ i,j\in \{1,2,3 \} $

		and the points are labeled as $p_{ij}=\ell_{1i}\cap \ell_{2j}$.

		Then by regarding the OLS of order 3 we find the incidence relations,		
		$$ p_{ij} = \ell_{1i}\cap \ell_{2j} \cap \ell_{3M_{ij}} \cap \ell_{4N_{ij}} $$	
		
\subsection{A Tropicalization Of The Possible (4,3)-Net}	

We use a method that is similar to the method explained in \textit{Section 3.3} to tropicalize the possible (4,3)-net. We find a transformation between the lines $\ell_{11}, \ell_{12}, \ell_{21}$ and $\ell_{22}$ and $z=0$, $x+y+z=0,\ x=0$ and $y=0$ respectively. Then we find the new locations of the points $p_{11}, p_{12}, p_{21}$ and $p_{22}$ after the transformation. 

	\begin{eqnarray*}
	p_{11} &=& \ell_{11} \cap \ell_{21} = (0:1:0)  \\
	p_{12} &=& \ell_{11} \cap \ell_{22} = (1:0:0)  \\ 
	p_{21} &=& \ell_{12} \cap \ell_{21} = (0:1:-1)   \\
	p_{22} &=& \ell_{12} \cap \ell_{22} = (1:0:-1)  
	\end{eqnarray*}
	
Also, we know the equations of the lines $\ell_{32}$ and $\ell_{41}$. Since $\ell_{32}$ passes through the points $p_{12}$ and $p_{21}$, the equation of $\ell_{32}$ is $y+z=0$. Similarly $\ell_{41}$ passes through $p_{11}$ and $p_{22}$, therefore its equation is $x+z=0$. Then $p_{33}$ would be $(1:1:-1)$, since it is at the intersection of $\ell_{32}$ and $\ell_{41}$.

\subsubsection{The Tropicalization Of The Lines and The Points}	
	We use the same matrix 
		$$
           T=\left[
             \begin{array}{ccc}
              t-t^2 & t^2-t^4 & t^4 \\
              t^3+t^2 & 1 & -t^3 \\
              t^2 & t^5 & 1\\
            \end{array}
          \right] $$
	as in \textit{Section 3.3.1} to tropicalize the lines. By \textit{Proposition 3.2} the matrix above has nonzero determinant except for finitely many values of t. The possible coordinates of the center of a line that is the transformation of $ax+by+cz=0$ is given in \textit{Lemma 3.1}. Also, the possible coordinates for a point that is the transformation of $p_{ij}=(a:b:c)$ is given in \textit{Lemma 3.3}.

\subsubsection{Determining Other Lines And Points}
After the tropicalization, the lines and the points are are as follows:
			$$
				\left.
					\begin{matrix}
						& \L_{11} & \L_{12} & \L_{21} & \L_{22} & \L_{31} & \L_{32} \\		
						\left[
						\begin{matrix} \\ \\ \end{matrix} \right. &
						\begin{matrix} -4 \\ -3 \end{matrix}&
						\begin{matrix} 4 \\ 3 \end{matrix} &
						\begin{matrix} 0 \\ -1 \end{matrix}&
						\begin{matrix} 1 \\ 5 \end{matrix} &
						\begin{matrix} 3 \\ 2 \end{matrix}&
						\begin{matrix} -2 \\ 0 \end{matrix} &
						\left.
						\begin{matrix} \\ \\ \end{matrix} \right] &		
					\end{matrix}
				\right.
			$$ 
			and 
			\begin{eqnarray*}
			P_{11} &=& (-2,-1) \\
			P_{12} &=& (1,2) \\
			P_{21} &=& (4,3) \\
			P_{22} &=& (1,3) \\
			P_{33} &=& (0,2)
 			\end{eqnarray*}

These are given in \textit{Figure 8}.		
			
\begin{figure}[h] \begin{center} \resizebox{8cm}{6cm} {\includegraphics{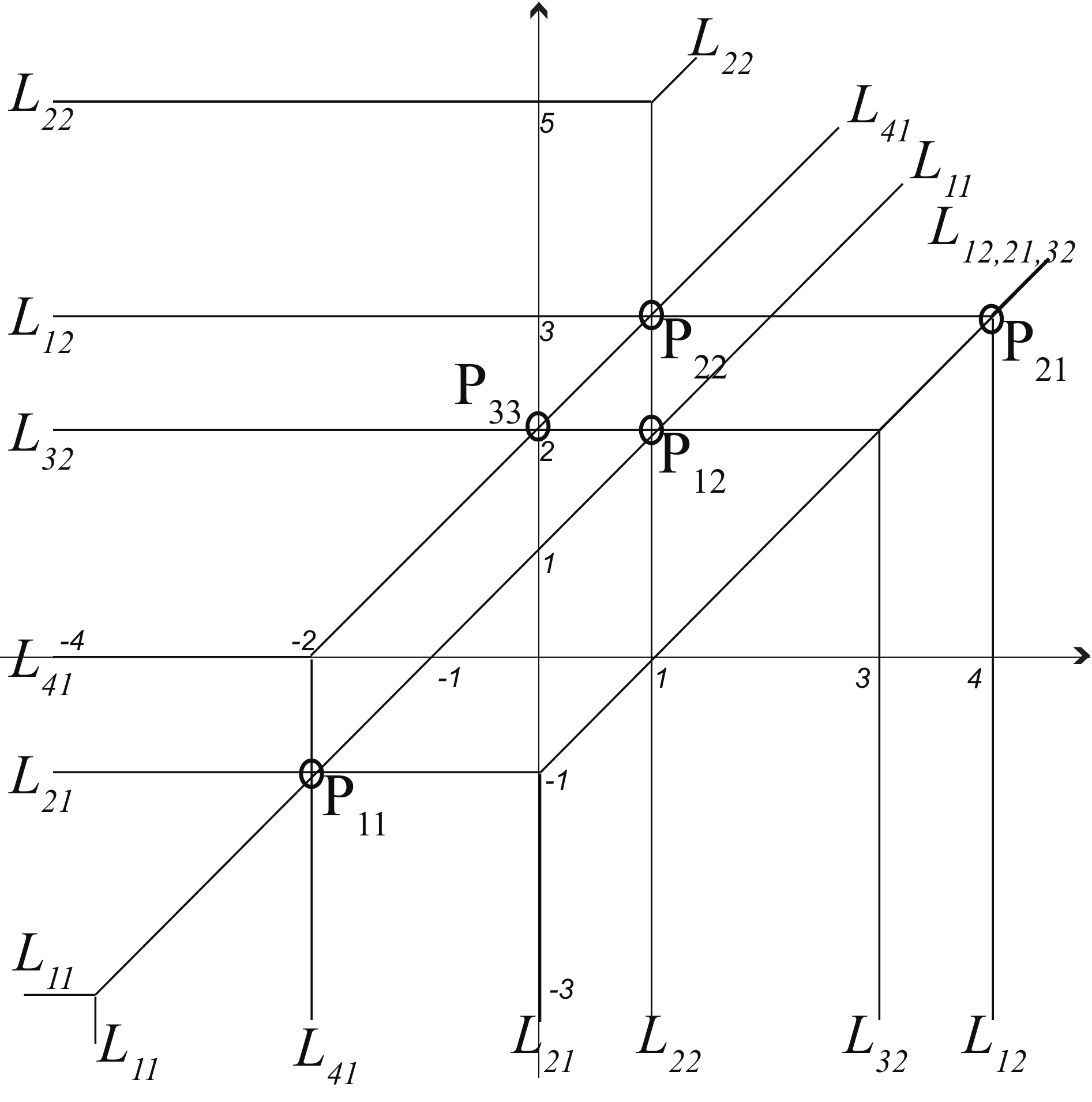}} \caption{Some of the lines and points of the (4,3)-net} \label{fig:fourthree} \end{center}\end{figure}

We want to determine the center of the line $\L_{31}$. The line $\L_{31}$ passes through $P_{11}$. The only possible centers for a line passing through $P_{11}$ are $(-2,-1), (-2,0), (0,-1)$, $(-4,-3)$. Considering the point-line table, if a line has a center at $(-2,0)$ then its equation is $x+z=0$, so the line is $\ell_{41}$. Similarly if a line has a center at $(0,-1)$ then its equation is $x=0$, so the line is $\ell_{21}$. If a line has a center at $(-4,-3)$ then its equation is $z=0$, so the line is $\ell_{11}$. Therefore the only possible center for $\L_{31}$ is $(-2,-1)$ and the equation of $\ell_{31}$ is $x+k_1z=0$ where $k_1\in \mathbb{Z}-\{0,1\}$.

Similarly considering $P_{22}$ on $\L_{33}$, we get that the center of $\L_{33}$ is $(1,3)$ and the equation of $\ell_{33}$ is $x+k_3y+z=0$ where $k_3\in \mathbb{Z}-\{0,1\}$.

If we make similar calculations for $\L_{43}$, by using that $P_{21}$ lies on $\L_{43}$ we get that the center of $\L_{43}$ is $(3,2)$ and the equation of $\ell_{43}$ is $k_2x+y+z=0$ where $k_2 \in \mathbb{Z}-\{0,1\}$.
\newpage

Up to now we know the following:

	\begin{eqnarray*}
		\ell_{11} &:& z = 0  \\
		\ell_{12} &:&  x+y+z=0 \\
		\ell_{21} &:& x=0 \\
		\ell_{22} &:& y=0 \\
		\ell_{31} &:& x+k_1z= 0 \\
		\ell_{32} &:& y+z = 0 \\
		\ell_{33} &:& x+k_3y+z = 0 \\
		\ell_{41} &:& x+z=0 \\
		\ell_{43} &:& k_2x+y+z=0 \\
	\end{eqnarray*}			
	\begin{eqnarray*}
		p_{11} &:& (0:1:0)  \\
		p_{12} &:&  (1:0:0) \\
		p_{21} &:& (0:1:-1) \\
		p_{22} &:& (1:0:-1) \\
		p_{33} &:& (1:1:-1) \\
	\end{eqnarray*}
	
	$$ k_1,k_2,k_3\in \mathbb{Z}-\{0,1\} $$

\newpage

We determine the other points and lines.

$\\ \underline{p_{13}} : $
	\begin{eqnarray}
		Let \ p_{13} = (x: y: z) \in \mathbb{P}^2  \nonumber \\
		p_{13} = \ell_{11}\cap \ell_{23}\cap \ell_{33}\cap \ell_{43} \nonumber \\
		z=0  \\
		x+k_3y+z=0 \\
		k_2x+y+z=0 \\
		By \ (4.1)\ and\ (4.2) \ if \ y=1 \Rightarrow x=-k_3 \nonumber \\
		So\ p_{13}=(-k_3: 1:0) \nonumber \\
		By \ (4.3) \ k_2k_3=1 \nonumber \\ 
		Since \ k_2\neq 0 \ k_3=\frac{1}{k_2} \\
		By\ (4.4)\ p_{13}=(-1:k_2:0) \nonumber
	\end{eqnarray}	
	
$\\ \underline{p_{32}} : $
	\begin{eqnarray}
		Let \ p_{32} = (x: y: z) \in \mathbb{P}^2  \nonumber \\
		p_{13} = \ell_{13}\cap \ell_{22}\cap \ell_{31}\cap \ell_{43} \nonumber \\
		y=0 \\
		x+k_1z=0 \\
		k_2x+y+z=0 \\
		By \ (4.6) \ if \ z=1 \Rightarrow x=-k_1 \nonumber \\
		So\ p_{13}=(-k_1: 0:1) \nonumber \\
		By \ (4.7) \ k_1k_2=1 \nonumber \\
		Since \ k_2\neq 0 \ k_1=\frac{1}{k_2} \\
		By \ (4.8)\ p_{32} = (-1:0:k_2)\nonumber 
	\end{eqnarray}	
	
$\\ \underline{p_{23}} : $
	\begin{eqnarray}
		Let \ p_{23} = (x: y: z) \in \mathbb{P}^2  \nonumber \\
		p_{23} = \ell_{12}\cap \ell_{23}\cap \ell_{31}\cap \ell_{42} \nonumber \\
		x+y+z=0 \\
		x+k_1z=0 \\
		By \ (4.9)\ and\ (10) \ if \ z=1 \Rightarrow x=-k_1,\ y=k_1-1 \nonumber \\
		So\ p_{23}=(-k_1: k_1-1:1) \nonumber \\
		By \ (4.8)\ p_{23}=(-1:1-k_2:k_2) \nonumber
	\end{eqnarray}	
	
$\\ \underline{p_{31}} : $
	\begin{eqnarray}
		Let \ p_{31} = (x: y: z) \in \mathbb{P}^2  \nonumber \\
		p_{31} = \ell_{13}\cap \ell_{21}\cap \ell_{33}\cap \ell_{42} \nonumber \\
		x=0 \\
		x+k_3y+z=0 \\
		By \ (4.7)\ and\ (4.8) \ if \ y=1 \Rightarrow z=-k_3 \nonumber \\
		So\ p_{31}=(0:1:-k_3) \nonumber \\
		By\ (4.4) \ p_{31}=(0:k_2:-1) \nonumber
	\end{eqnarray}	

$\\ \underline{\ell_{13}} : $
	\begin{eqnarray}
		Let \ \ell_ {13} = [x: y: z] \in (\mathbb{P}^2)^*  \nonumber \\
		p_{31}, p_{32}, p_{33} \in \ell_{13} \nonumber \\
		k_2y-z=0 \\
		-x+k_2z=0 \\
		x+y-z=0 \\
		By \ (4.13), \ if \ y=1 \Rightarrow z=k_2,\ by\  (4.14), \  x=k_2^2 \nonumber \\
		So\ \ell_{4.13}=[k_2^2: 1: k_2 ] \nonumber \\
		By \ (4.15) \ k_2^2-k_2+1=0
	\end{eqnarray}	

$\\ \underline{\ell_{23}} : $
	\begin{eqnarray}
		Let \ \ell_{23} = [x: y: z] \in (\mathbb{P}^2)^*  \nonumber \\
		p_{13}, p_{23}, p_{33} \in \ell_{23} \nonumber \\
		-x+k_2y=0 \\
		-x+(1-k_2)y+k_2z=0 \\
		x+y-z=0 \\
		By \ (4.17), \ if \ y=1 \Rightarrow x=k_2,\ by\  (4.19), \  z=k_2+1 \nonumber \\
		So\ \ell_{23}=[k_2: 1: k_2+1 ] \nonumber \\
		(4.18)\ is \ satisfied\ since \ k_2^2-k_2+1=0 \nonumber
	\end{eqnarray}	

$\\ \underline{\ell_{42}} : $
	\begin{eqnarray}
		Let \ \ell_{42} = [x: y: z] \in (\mathbb{P}^2)^*  \nonumber \\
		p_{12}, p_{23}, p_{31} \in \ell_{42} \nonumber \\
		x=0 \\
		-x+(1-k_2)y+k_2z=0 \\
		k_2y-z=0 \\
		By \ (4.22), \ if \ y=1 \Rightarrow z=k_2 \nonumber \\
		So\ \ell_{42}=[0: 1: k_2] \nonumber \\
		(4.21)\ is \ satisfied\ since \ k_2^2-k_2+1=0 \nonumber
	\end{eqnarray}	
	
The equation has two solutions $\frac{1\mp \sqrt{-3}}{2}$, so we obtain two (4,3)-nets. Complex conjugation gives us an isomorphism between these two. Hence up to isomorphism there exists a unique (4,3)-net. The lines and the points of the (4,3)-net are given below:

	\begin{eqnarray*}
		\ell_{11} &:& z = 0  \\
		\ell_{12} &:&  x+y+z=0 \\
		\ell_{13} &:& k_2^2x+y+k_2z=0 \\ 
		\ell_{21} &:& x=0 \\
		\ell_{22} &:& y=0 \\
		\ell_{23} &:& k_2x+y+(k_2+1)z=0 \\
		\ell_{31} &:& k_2x+z= 0 \\
		\ell_{32} &:& y+z = 0 \\
		\ell_{33} &:& k_2x+y+k_2z = 0 \\
		\ell_{41} &:& x+z=0 \\
		\ell_{42} &:& y+k_2z=0 \\
		\ell_{43} &:& k_2x+y+z=0 \\
	\end{eqnarray*}			
	\begin{eqnarray*}
		p_{11} &:& (0:1:0)  \\
		p_{12} &:&  (1:0:0) \\
		p_{13} &:& (-1:k_2:0) \\
		p_{21} &:& (0:1:-1) \\
		p_{22} &:& (1:0:-1) \\
		p_{23} &:& (-1:1-k_2:k_2) \\
		p_{31} &:& (0:k_2:-1) \\
		p_{32} &:& (-1:0:k_2) \\
		p_{33} &:& (1:1:-1) \\
	\end{eqnarray*}
	
	where $k_2^2-k_2+1=0$

\end{document}